\documentclass{amsart}
\usepackage[utf8]{inputenc}
\usepackage{amsmath}
\usepackage{amsfonts}
\usepackage{amssymb}
\usepackage{amsthm}
\usepackage{enumitem}
\usepackage[english]{babel}
\usepackage{booktabs}

\newtheorem{theorem}{Theorem}[section]
\newtheorem{proposition}[theorem]{Proposition}
\newtheorem*{theorem*}{Theorem}
\newtheorem*{proposition*}{Proposition}
\theoremstyle{remark}
\newtheorem{remark}[theorem]{Remark}
\theoremstyle{definition}

\newcommand{\Oc}{\operatorname{\mathcal{O}}} 
\newcommand{\Fd}{\operatorname{\mathbb{F}}} 
\newcommand{\Z}{\operatorname{\mathbb{Z}}} 
\newcommand{\Q}{\operatorname{\mathbb{Q}}} 
\newcommand{\Proj}{\operatorname{\mathbb{P}}} 
\newcommand{\Ql}{\operatorname{\mathbb{Q}_{\ell}}} 
\newcommand{\Qal}{\operatorname{\overline{\mathbb{Q}}}} 
\DeclareMathOperator{\Gal}{Gal} 
\DeclareMathOperator{\Char}{char} 
\newcommand{\Tl}{\ensuremath{T}_{\ell}} 
\newcommand{\Vl}{\ensuremath{V}_{\ell}} 
\newcommand{\Zl}{\operatorname{\mathbb{Z}_{\ell}}} 
\DeclareMathOperator{\Aut}{\operatorname{Aut}} 
\DeclareMathOperator{\Fil}{Fil} 
\newcommand{\MF}{\operatorname{\mathbf{MF}}} 
\newcommand{\Rep}{\operatorname{\mathbf{Rep}}} 
\newcommand{\Hom}{\operatorname{\mathbf{Hom}}} 
\newcommand{\Bcris}{\operatorname{\mathbf{B}_{cris}}} 
\newcommand{\D}{\mathbf{D}} 
\newcommand{\Et}{\ensuremath{\tilde{E}}} 
\newcommand{\Dcris}[1][]{\mathbf{D}^{*}_{\mathrm{cris},\ensuremath{#1}}} 
\DeclareMathOperator{\dst}{\mathbf{dst}} 
\DeclareMathOperator{\detm}{\mathrm{det}} 
\newcommand{\Lg}{\ensuremath{L}^{\mathrm{a}}} 
\newcommand{\Lng}{\ensuremath{L}^{\mathrm{na}}} 
\DeclareMathOperator{\Dc}{\D_{c}} 
\DeclareMathOperator{\Dpc}{\D_{pc}} 
\DeclareMathOperator{\Dpcg}{\D_{pc}^{a}} 
\DeclareMathOperator{\Dpcng}{\D_{pc}^{na}} 
\DeclareMathOperator{\Tr}{Tr} 
\DeclareMathOperator{\Id}{Id} 
\DeclareMathOperator{\W}{\mathbf{W}} 

\renewcommand{\restriction}{\mathord{\upharpoonright}}

\let\temp\phi{}
\let\phi\varphi{}
\let\varphi\temp{}

\title[$3$-adic representations from elliptic curves over $\Q_3$ with potential good reduction]{The $3$-adic representations
arising from elliptic curves over $\Q_3$ with potential good reduction}
\author{Giovanni Bosco}
\subjclass[2020]{11G07, 11F80, 11F85}
\keywords{Galois representations, elliptic curves, $p$-adic Hodge-theory} 
\address{Université de Mons, Département de Mathématiques, 7000 Mons, Belgique}
\email{giovanni.bosco@umons.ac.be}

\begin{document}

\begin{abstract}
    We give a complete classification of all the potentially crystalline $3$-adic representations of the absolute Galois group
    of $\Q_3$ that are isomorphic to the Tate module of an elliptic curve defined over $\Q_3$. These representations are
    described in terms of their associated filtered $(\phi,\Gal({K}/{\Q_3}))$-modules. The most interesting cases occur when
    the potential good reduction is wild.
\end{abstract}

\maketitle

\tableofcontents

\newpage

\section{Introduction}

The $p$-adic representations arising from elliptic curves over $\Q_p$ have been completely described for $p\ge 5$ in~\cite{Vo1}. The goal of this paper is to treat the case of potential good reduction for $p=3$. When $p\ge 5$ potential good reduction is necessarily tame with cyclic inertia. This is not the case anymore for $p=3$, where both wild potential good reduction and non abelian inertia do appear, sometimes simulteanously (see $e=12$).

\vspace{\baselineskip}

Let $\Qal_p$ be an algebraic closure of $\Q_p$ and $G=\Gal({\Qal_p}/{\Q_p})$ its absolute Galois group.
Given an elliptic curve ${E}$ defined over ${\Q_p}$, let $E[p^n]$ denote its group of $p^n$-torsion points with value in $\Qal_p$ and
\[
T_p(E)=\underset{P\mapsto pP}{\varprojlim} E[p^n]
\]
its $p$-adic Tate module. It is a free $\Z_p$-module of rank $2$ with a continuous and linear action of $G$.
The $p$-adic representation of $G$ associated to $E$ (also called Tate module) is
\[
V_p(E)=\Q_p\otimes_{\Z_p}T_p(E).
\]
A $p$-adic representation $V$ of $G$ arises from an elliptic curve over $\Q_p$ if there exists ${E}/{\Q_p}$ such
that $V\simeq V_p(E)$. We wish to classify all $p$-adic representations arising from elliptic curves over $\Q_p$ up to isomorphism for $p=3$ with an additional condition:
the considered elliptic curve has potential good reduction, that is, it acquires good reduction over a finite extension
of $\Q_p$. Such curves have nice geometric properties which are carried over the representations. Indeed, it is well known that the Tate module of an elliptic curve with potential good reduction is potentially crystalline.
Such representations are completely determined --- via the contravariant functor $\mathbf{D}_{\mathrm{pcris}}^{*}$ --- by their
associated filtered $(\phi,\Gal({K}/{\Q_p}))$-module, a purely semilinear object.

\vspace{\baselineskip}

Let ${E}/{\Q_p}$ be an elliptic curve acquiring good reduction over a finite Galois extension ${K}/{\Q_p}$ with maximal unramified
subfield $K_0$ such that its ramification index $e=e({K}/{\Q_p})$ is minimal. Let $\D=(D,\Fil)$ be its associated filtered $(\phi,\Gal({K}/{\Q_p}))$-module,
$D_0$ the subspace of elements fixed by $\Gal({K_0}/{\Q_p})$ and $\phi_0=\phi\restriction_{D_0}$ the $\Q_p$-linear restriction
of $\phi$. We denote by $\W(\D)$ the Weil representation associated to $\D$. It is known that $\D$ satisfies the following
properties:

\vspace{\baselineskip}

\begin{enumerate}[label= (\arabic*)]
	\item $P_{\Char}(\phi_0)(X)=X^2+a_{p}X+p$, with ${\vert a_p\vert}_{\infty}\le 2\sqrt{p}$
	\item $\W(\D)$ is defined over $\Q$
	\item $\bigwedge^2_{K_0}\D=K_0\{-1\}$ (i.e. $\bigwedge_{\Q_p}^2V_p(E)=\Q_p(1)$)
	\item $\D$ is of Hodge-Tate type $(0,1)$.
\end{enumerate}

\vspace{\baselineskip}

These conditions alone are sufficient to guarantee that a $2$-dimensional $p$-adic representation of $G$ comes from an elliptic curve
over $\Q_p$ in the case of tame potential good reduction (see~\cite{Vo1},Thm.5.1.\ or~\cite{Vo2}, \S 5.4). It is not known yet if these are
sufficient in the presence of wild potential good reduction as well, however they are still necessary.
Starting from these conditions and imposing geometric descent datum and a minimal field of good reduction $K$, we provide a list of isomorphism classes of possible
filtered $(\phi,\Gal({K}/{\Q_3}))$-modules. Then we show that every object in the list arises from an elliptic curve over $\Q_3$.

\vspace{\baselineskip}

Some of the classes described in this paper can directly be deduced from the $p\ge 5$ case when $(e,p)=1$ (see~\cite{Vo2}).
To the best of our knowledge, a complete classification
of $\ell$-adic representations ($\ell\neq 3$) --- which is encoded in terms of unfiltered $(\phi,\Gal({K}/{\Q_3}))$-modules --- does not appear in the litterature. However some particular cases can be found (see~\cite{Co} for $e=12$).
Our genuine new results are the cases of wild potential good reduction ($e=3,6$ and $12$) with $e=12$ being the first case
of non abelian inertia. We provide proofs in the tame case for the sake of completeness. The classification is synthetized
in Table~\ref{tab:Table 1}. Notations for the filtered $(\phi,\Gal({K}/{\Q_3}))$-modules and their set of parameters are detailed in section~\ref{S4}.

\vspace{\baselineskip}

\renewcommand{\arraystretch}{1.3}
\begin{table}[h]
\begin{center}
\begin{tabular}{|c|c|c|c|l|c|}
  \hline
  $e$ & Reduction type & $K$ & Frobenius & Filtered $(\phi,\Gal({K}/{\Q_3}))$-module & \#Classes\\
  \hline
  $1$ & Supersingular & $\Q_3$ & $a_3=-3$ & $\Dc(1;-3;0)$ & $1$ \\
  \cline{4-6}
  &  &  & $a_3=0$ & $\Dc(1;0;0)$ & 1 \\
  \cline{4-6}
  &  &  & $a_3=3$ & $\Dc(1;3;0)$ & 1 \\
  \cline{2-6}
  & Ordinary & $\Q_3$ & $a_3=-2$ & $\Dc(1;-2;\alpha)$, $\alpha\in\{0,1\}$ & 2\\
  \cline{4-6}
  & & & $a_3=-1$ & $\Dc(1;-1;\alpha)$, $\alpha\in\{0,1\}$ & 2 \\
  \cline{4-6}
  & & & $a_3=1$ & $\Dc(1;1;\alpha)$, $\alpha\in\{0,1\}$ & 2 \\
  \cline{4-6}
  & & & $a_3=2$ & $\Dc(1;2;\alpha)$, $\alpha\in\{0,1\}$ & 2 \\
  \hline
  $2$ & Supersingular & $\Q_3(\sqrt{3})$ & $a_3=-3$ & $\Dc(2;-3;0)$ & 1\\
  \cline{4-6}
  &  &  & $a_3=0$ & $\Dc(2;0;0)$ & 1 \\
  \cline{4-6}
  &  &  & $a_3=3$ & $\Dc(2;3;0)$ & 1 \\
  \cline{2-6}
  & Ordinary & $\Q_3(\sqrt{3})$ & $a_3=-2$ & $\Dc(2;-2;\alpha)$, $\alpha\in\{0,1\}$ & 2\\
  \cline{4-6}
  & & & $a_3=-1$ & $\Dc(2;-1;\alpha)$, $\alpha\in\{0,1\}$ & 2 \\
  \cline{4-6}
  & & & $a_3=1$ & $\Dc(2;1;\alpha)$, $\alpha\in\{0,1\}$ & 2 \\
  \cline{4-6}
  & & & $a_3=2$ & $\Dc(2;2;\alpha)$, $\alpha\in\{0,1\}$ & 2 \\
  \hline
  $4$ & Supersingular & $\Q_3(\sqrt[4]{3})$ &  $a_3=0$ & $\Dpc(4;0;\alpha), \alpha\in\Proj^1(\Q_3) $ & $\Proj^1(\Q_3)$ \\
  \hline
  $3$ & Supersingular & $\Lng(\zeta_4)$ & $a_3=0$ & $\Dpcng(3;0;\alpha), \alpha\in \mathcal{M}_3^{\mathrm{na}}$ & $\Proj^1(\Q_3)$ \\
  \cline{3-6}
  & & $\Lg=\Q_3(\pi)$ & $a_3=-3$ & $\Dpcg(3;-3,\mu;\pi), \mu\in\{1,2\}$ & 2 \\
  \cline{4-6}
  & & & $a_3=0$ & $\Dpcg(3;0,\mu;\pi), \mu\in\{-1,1\}$ & 2 \\
  \cline{4-6}
  & & & $a_3=3$ & $\Dpcg(3;3,\mu;\pi), \mu\in\{-2,-1\}$ & 2 \\
  \hline
  $6$ & Supersingular & $\Lng(\zeta_4,\sqrt{3})$ & $a_3=0$ & $\Dpcng(6;0;\alpha), \alpha\in \mathcal{M}_6^{\mathrm{na}}$ & $\Proj^1(\Q_3)$ \\
  \cline{3-6}
  & & $\Lg(\sqrt{3})$ & $a_3=-3$ & $\Dpcg(6;-3,\mu;\pi), \mu\in\{1,2\}$ & 2 \\
  \cline{4-6}
  & & & $a_3=0$ & $\Dpcg(6;0,\mu;\pi), \mu\in\{-1,1\}$ & 2 \\
  \cline{4-6}
  & & & $a_3=3$ & $\Dpcg(6;3,\mu;\pi), \mu\in\{-2,-1\}$ & 2 \\
  \hline
  $12$ & Supersingular & $K_1$ & $a_3=0$ & $\Dpc(12;0;1;\epsilon;\alpha), \alpha\in\mathcal{M}_{12}^{1,\epsilon}, \epsilon\in\{0,1\}$ & $\Proj^1(\Q_3)$ \\
  \cline{3-6}
  & & $K_2$ & $a_3=0$ & $\Dpc(12;0;2;\epsilon;\alpha), \alpha\in \mathcal{M}_{12}^{2,\epsilon}, \epsilon\in\{0,1\}$ & $\Proj^1(\Q_3)$ \\
  \cline{3-6}
  & & $K_3$ & $a_3=0$ & $\Dpc(12;0;3;\epsilon;\alpha), \alpha\in \mathcal{M}_{12}^{3,\epsilon}, \epsilon\in\{0,1\}$ & $\Proj^1(\Q_3)$ \\
  \cline{3-6}
  & & $K_4$ & $a_3=0$ & $\Dpc(12;0;4;\epsilon;\alpha), \alpha\in \mathcal{M}_{12}^{4,\epsilon}, \epsilon\in\{0,1\}$ & $\Proj^1(\Q_3)$ \\
  \cline{3-6}
  & & $K_5$ & $a_3=0$ & $\Dpc(12;0;5;\epsilon;\alpha), \alpha\in \mathcal{M}_{12}^{5,\epsilon}, \epsilon\in\{0,1\}$ & $\Proj^1(\Q_3)$ \\
  \hline
\end{tabular}
\end{center}
\caption{\label{tab:Table 1}Isomorphism classes of filtered $(\phi,\Gal({K}/{\Q_3}))$-modules arising from elliptic curves over $\Q_3$ with potential good reduction.}
\end{table}

Note that the supersingular traces $a_3=\pm 3$ occur, which is specific to the $p=3$ case (compared to $p\ge 5$). One may expect that they should appear every time the reduction is supersingular, and yet this is not the case.
The reason behind this absence lies in the structure of the automorphism group of the special fibre, which controls the possible descents. Furthermore, we need to deal with several different fields of good reduction.
Indeed, wild finite extensions of $\Q_3^{\mathrm{un}}$ aren't unique as opposed to the tame ones. This leads to interesting
new phenomena. The case $e=12$ is uniform, the five fields are almost indistinguishable. When $e=3$ the situation is different
between the two possible fields. The non abelian extension occurs for only one possible Frobenius trace and has an infinity of
isomorphism classes. The abelian extension, on the other hand, occurs for every supersingular traces value but has only two classes for each.
Let us finally mention that the ordinary cases have simply disappeared when $e>2$, again a specific feature of elliptic curves
over $\Fd_3$.

\section{Theoretical background}

Let $G_{\Q_p}=\Gal({\Qal_p}/{\Q_p})$ the absolute Galois group of $\Q_p$. We denote by $\Q_p^{\mathrm{un}}$ its maximal unramified
extension and $I_{\Q_p}=\Gal({\Qal_p}/{\Q_p^\mathrm{un}})$ its inertia subgroup.

\subsection{Elliptic curves}

Let ${E}/{\Q_p}$ be an elliptic curve (we refer to~\cite{Si} for the arithmetic of elliptic curves).
One may assume, after a suitable change of coordinates, that the coefficients of a Weierstrass equation of $E$ are in $\Z_p$ and
that the valuation of its discriminant is minimal. A Weierstrass equation satisfying these two properties is called minimal.
Suppose $E$ is given by a minimal Weierstrass equation, reducing each coefficient we obtain a curve ${\Et}/{\Fd_p}$.
The reduced curve need not be an elliptic curve itself, in fact it will be if and only if
$v(\Delta)=0$ (i.e. $\Delta(\Et)=\Delta\bmod p\Z_p\neq 0$). When the reduced curve is an elliptic curve we say that $E$ has
good reduction. Let ${L}/{\Q_p}$ be a finite extension of $\Q_p$ and consider $E_L=E\times_{\Q_p}L$ the extension of $E$ to $L$.
Allowing changes of coordinates defined over $L$ may give us a minimal model of $E_L$ with $v(\Delta_L)=0$, so that $E_L$ has
good reduction. When there exists such an extension we say that $E$ has potential good reduction. This property only depends
on the action of inertia, which means we can choose ${L}/{\Q_p}$ to be totally ramified so ${\Et_L}/{\Fd_p}$. Denote
by $a_p(E)=a_p(\Et_L)$ the trace of the characteristic polynomial of the Frobenius endomorphism acting on $\Vl(\Et_L)$
for some $\ell\neq p$. It is known that $a_p(E)$ is an integer independent of $\ell$
satisfying $\vert a_p(E)\vert_{\infty}\le 2\sqrt{p}$ as well as an invariant of the isogeny class of $\Et_L$ over $\Fd_p$
(see~\cite{Ho-Ta}). Furthermore we have the following relation:
\[
a_p(\Et_L)=p+1-\# \Et_L(\Fd_p).
\]
We say that $\Et_L$ is ordinary when $(p,a_p(\Et_L))=1$, supersingular when $p\mid a_p(\Et_L)$.

\subsection{$\ell$-adic Galois representations}

Let $p,\ell$ be distinct prime numbers. An $\ell$-adic representation of $G_{\Q_p}$
(or $\Ql[G_{\Q_p}]$-module) is a finite dimensional $\Ql$-vector space with a linear and continuous action of $G_{\Q_p}$.
We denote such an object by $(V,\rho_{\ell})$ where $V$ is a $\Ql$-vector space
and $\rho_{\ell}:G_{\Q_p}\longrightarrow \Aut_{\Ql}(V)$ the group homomorphism describing the action.
If the inertia subgroup $I_{\Q_p}$ of $G_{\Q_p}$ acts trivially on $V$ we say that the representation has good reduction.
In this case it factors into a representation of the absolute Galois group $G_{\Fd_p}$ of $\Fd_p$ and is completely determined
by it. When there exists a finite extension ${L}/{\Q_p}$ such that $I_L$ acts trivially on $V$ we say that the representation
has potential good reduction. One easily checks that having potential good reduction is equivalent to $\rho_{\ell}(I_L)$
being finite. Let ${E}/{\Q_p}$ be an elliptic curve, the group $G_{\Q_p}$ acts on $E(\Qal_p)$ by acting on the coefficients of its 
points.
Since addition is $G_{\Q_p}$-equivariant, the group of $n$-torsion points $E[n]$ of $E(\Qal_p)$ is stable by action of $G_{\Q_p}$ 
and we define the $\ell$-adic Tate module associated to $E$ by
\[
\Tl(E)=\varprojlim_{P\mapsto\ell P} E[\ell^n].
\]
It is a free $\Zl$-module of rank $2$ equipped with a continuous and $\Zl$-linear action of $G_{\Q_p}$.
Tensoring by $\Ql$ we get $\Vl(E)=\Ql\otimes\Tl(E)$, an $\ell$-adic representation of $G_{\Q_p}$. It is well known
that $\Vl(E)$ has (potential) good reduction if and only if $E$ has (potential) good reduction. If ${E}/{\Q_p}$ is an elliptic curve
with potential good reduction, there exists a unique finite extension $M_{E}/{\Q_p}^{\mathrm{un}}$ of minimal degree over
which $E$ acquires good reduction. We call that minimal degree the semi-stability defect of $E$, denoted by $\dst(E)$.
Consider $\Vl(E)$ the $\ell$-adic representation associated to $E$, since $E$ has potential good reduction there
exists ${L}/{\Q_p}$ finite of minimal ramification index satisfying $\rho_{E,\ell}(I_L)=0$, it is then easy to see that
\[
M_E=\Qal_p^{\ker(\rho_{E,\ell}\restriction_{I_L})}.
\]
If ${L}/{\Q_p}$ is a finite extension with $L^{\mathrm{un}}=M_E$ then $E$ acquires good reduction over $L$ and $\dst(E)=e({L}/{\Q_p})$,
it is the minimal ramification index among all good reduction fields of $E$. It is also worth noticing that
if $L,{L^{\prime}}/{\Q_p}$ satisfy $L^{\mathrm{un}}={(L^{\prime})}^{\mathrm{un}}$, then they are interchangeable in the sense
that ${E}/{\Q_p}$ acquires good reduction over $L$ if and only if it acquires good reduction over $L^{\prime}$. Furthermore,
we know that $\dst(E)\in\{1,2,3,4,6,12\}$ and $\rho_{E,\ell}(I_{\Q_p})$ is either a cyclic group of order $1,2,3,4,6$ or the
non Abelian semi-direct product of a cyclic group of order $4$ by a group of order $3$ (see~\cite{Se1},\S 5.6). The degree of a
minimal good reduction field is bounded by the image of inertia and the structure of its inertia subgroup is known.

\subsection{Filtered $(\phi,\Gal({K}/{\Q_p}))$-modules}

Let ${K}/{\Q_p}$ be a finite Galois extension, $K_0$ the maximal unramified extension of $\Q_p$ inside $K$ and $G_K=\Gal(\Qal_p/K)$
its absolute Galois group. Denote by $\sigma$ the absolute Frobenius on $K_0$. A filtered
$(\phi,\Gal({K}/{\Q_p}))$-module $\D=(D,\Fil)$ is a finite dimensional $K_0$-vector space $D$ together with:
    \begin{enumerate}[label= (\roman*)]
        \item a $\sigma$-semilinear action of $\Gal({K}/{\Q_p})$
        \item a $\sigma$-semilinear, $\Gal({K}/{\Q_p})$-equivariant and bijective Frobenius $\phi:D\tilde{\longrightarrow}D$
        \item a filtration $\Fil={(\Fil^{i}D_K)}_{i\in\Z}$ on $D_K=K\otimes_{K_0}D$ by $\Gal({K}/{\Q_p})$-stable subspaces
        such that $\Fil^{i}D_K=D_K$ for $i\ll 0$ and $\Fil^{i}D_K=0$ for $i\gg 0$.
    \end{enumerate}

Such objects form a category we will denote by $\MF_{\phi}(G_{\Q_p})$. The morphisms are the $K_0$-linear maps $f$ commuting
to the Frobenius and the action of $\Gal({K}/{\Q_p})$ as well as preserving the filtration
(i.e. $f_K(\Fil^{i}D_K)\subseteq\Fil^{i}D^{\prime}_K$). The Tate twist $\D\{-1\}$ of $\D$ is the $K_0$-vector space $D$ with
the same action of $\Gal({K}/{\Q_p})$, $\phi\{-1\}=p\phi$ and $\Fil^i{(D\{-1\})}_K=\Fil^{i-1}D_K$. We say that $\D$ is
of Hodge-Tate type $(0,1)$ if $\Fil^{i}D_K=D_K$ for $i\le 0$, $\Fil^{i}D_K=0$ for $i\ge 2$ and $\Fil^1D_K$ is
a non trivial subspace of $D_K$. We associate to $\D$ the following quantities:
\begin{enumerate}
    \item $t_{N}(\D)=v_{p}(\det\phi)$
    \item $t_{H}(\D)=\underset{i\in\Z}{\sum}i\dim_{K}(\Fil^{i}D_K/\Fil^{i+1}D_K)$,
\end{enumerate}
where $\det\phi$ is the determinant of a matrix representing $\phi$. We say that $\D$ is admissible if $t_{N}(\D)=t_{H}(\D)$ and for every subobject $\D^{\prime}$ of $\D$, $t_{H}(\D^{\prime})\le t_{N}(\D^{\prime})$.
Let $V$ be a $p$-adic representation of $G_{\Q_p}$, one can associate to $V$ a
filtered $(\phi,\Gal({K}/{\Q_p}))$-module via the contravariant functor:
\[
\Dcris[K]:\Rep_{\Q_p}(G_{\Q_p}) \longrightarrow \MF_{\phi}(\Gal({K}/{\Q_p})): V\longmapsto \Hom_{\Q_p[G_K]}(V,\Bcris).
\]
Where $\Bcris$ is the crystalline period ring (see~\cite{Fo1}).
The inequality $\dim_{\Q_p}V\le\dim_{K_0}\Dcris[K](V)$ is always satisfied, and we say that a representation $V$ of $G_{\Q_p}$ is
crystalline over $K$ if the equality holds. Viewing $V$ as a representation of $G_K$ by restriction, then $V$ is
potentially crystalline over $K$ as a representation of $G_{\Q_p}$ if and only if it is crystalline as a representation
of $G_K$. This functor establishes an anti-equivalence of categories between the category of $p$-adic   representations of
$G_{\Q_p}$ crystalline over $K$ and the category of admissible filtered $(\phi,\Gal({K}/{\Q_p}))$-modules (see~\cite{Fo1}). 
The $p$-adic Tate modules of elliptic curves over $\Q_p$ with potential good reduction give rise to such representations. In fact,
the following holds:
\begin{theorem*}[\cite{Co-Io}, Thm.4.7]
Let ${E}/{\Q_p}$ be an elliptic curve,
the $p$-adic representation $V_p(E)$ is (potentially) crystalline if and only if $E$ has (potential) good reduction. 
\end{theorem*}
Each filtered $(\phi,\Gal({K}/{\Q_p}))$-module has a linear object naturally attached to it, namely its Weil representation.
Recall that the Weil group of $\Qal_p$ is defined by the short exact sequence
\[
1\longrightarrow I_{\Q_p}\longrightarrow W_{\Q_p} \overset{\nu}{\longrightarrow} \Z\longrightarrow 1
\]
and we let $W_K=G_K\cap W_{\Q_p}$. To every $(\phi,\Gal({K}/{\Q_p}))$-module $\D$ we can associate
a $K_0$-vector space $\Delta$ with a continuous $K_0$-linear action of $W_{\Q_p}$ in the following way:
\[
\rho: W_{\Q_p}\longrightarrow \Aut_{K_0}(\Delta): \omega\longmapsto (\omega\bmod W_K)\phi^{-\nu(\omega)}
\]
where $\Delta$ is the underlying $K_0$-vector space of $\D$. The pair $\W(\D)=(\Delta,\rho)$ is called a Weil representation.
It is defined over $\Q$ if $\Tr(\rho(w))\in\Q$ for every $w\in W_{\Q_p}$.

\section{Strategy}

We begin by fixing a semi-stability defect $e\in\{ 1,2,3,4,6,12\}$. The first step is to determine every finite Galois extension
with ramification index $e$ that arises as a field of good reduction of some elliptic curve defined over $\Q_3$. Cases $e=1,2$ and $4$ are tame, hence necessarily given by $\Q_3(\sqrt[e]{3})$. For $e=3$ we use Local Class Field Theory and the local fields database in~\cite{LMFDB}, $e=6$ is then obtained by a ramified quadratic twist. Finally, $e=12$ is treated using~\cite{LMFDB} since the structure of the Galois group and inertia subgroup are well known.
We then fix ${K}/{\Q_3}$ to be one such extension. The next step is to describe the list of the $2$-dimensional filtered
$(\phi,\Gal ({K}/{\Q_3}))$-modules $\D$ satisfying properties $(1)-(4)$. We then show that given an elliptic curve ${E}/{\Q_3}$ with
potential good reduction over $K$, its associated filtered $(\phi,\Gal ({K}/{\Q_3}))$-module $\Dcris[K] (V_3 (E))$ is necessarily
isomorphic to one object $\D$ of our list. Finally, given an object $\D$ in the list, we need to find an elliptic curve ${E}/{\Q_3}$
such that
\[
    \Dcris[K] (V_3 (E))\simeq\D,
\]
this is done in section~\ref{S5}.

One last point require some discussion. Given an unfiltered $2$-dimensional $(\phi,\Gal ({K}/{\Q_3}))$-module $\D$,
the set of Hodge-Tate type $(0,1)$ filtrations on $\D$ is in bijection with $\Proj^1(\Q_3)$. Indeed, by Galois Descent, it is easy to check that the $\Gal ({K}/{\Q_3})$-stable lines in $D_K=K\otimes_{K_0}D$ are in bijection with the lines in
\[
D_K^{\Gal({K}/{\Q_3})}=\{ x\in D_K\vert \forall g \in \Gal({K}/{\Q_3}),g.x=x\}
\]
a $2$-dimensional $\Q_3$-vector space. This means that if $\phi$ has only trivial stable subspaces, there are infinitely
many admissible filtrations on $\D$. In the following we will define sets that parameter our filtrations. This
fact ensures that these sets will always be non empty, even though it could not be clear at first glance.

\section{Classification}\label{S4}

We provide the list of admissible filtered $(\phi,\Gal({K}/{\Q_3}))$-modules satisfying our geometric conditions:

\begin{enumerate}[label= (\arabic*)]
	\item $P_{\Char}(\phi_0)(X)=X^2+a_{p}X+p$, with ${\vert a_p\vert}_{\infty}\le 2\sqrt{p}$
	\item $\W(\D)$ is defined over $\Q$
	\item $\bigwedge^2_{K_0}\D=K_0\{-1\}$ (i.e. $\bigwedge_{\Q_3}^2V_3(E)=\Q_3(1)$)
	\item $\D$ is of Hodge-Tate type $(0,1)$.
\end{enumerate}

with ${K}/{\Q_3}$ a
minimal Galois extension of good reduction. We then show that every elliptic curve defined over $\Q_3$ with potential
good reduction has associated filtered $(\phi,\Gal({K}/{\Q_3}))$-module isomorphic to an object of the list.

\subsection{\for{toc}{The crystalline case}\except{toc}{The crystalline case ($e=1$)}}

We start our classification with the representations coming from elliptic curves ${E}/{\Q_3}$ with good reduction ($K=\Q_3$).
There are two distinct cases behaving differently depending on the trace $a_3(\Et)$ of the Frobenius of ${\Et}/{\Fd_3}$.

\subsubsection{The supersingular case}

Let $a\in\{-3,0,3\}$ and $\alpha\in\Proj^1(\Q_3)$. We denote by $\Dc(1;a;\alpha)$ the filtered $\phi$-module
(of Hodge-Tate type $(0,1)$) defined by:
\begin{itemize}
	\item $D=\Q_3 e_1\oplus\Q_3 e_2$
	\item $M_B(\phi)=\begin{pmatrix} 0 & -3 \\ 1 & -a \end{pmatrix}$, where $B=(e_1,e_2)$
	\item $\Fil^1D=(\alpha e_1 + e_2)\Q_3$.
\end{itemize}
Identifying $\Proj^1(\Q_3)$ with $\Q_3\sqcup\{\infty\}$, we let $\alpha e_1 +e_2=e_1$ when $\alpha=\infty$. For each $a\in\{-3,0,3\}$ and each $\alpha\in\Proj^1(\Q_3)$, the filtered $\phi$-module $\Dc(1;a;\alpha)$ satisfies conditions $(1)-(4)$ and is admissible. Condition $(1)$ is obvious and $(4)$ is satisfied by definition. Conditions $(2)$ and $(3)$ as well as admissibility are easily checked by computation.

\begin{proposition}
Let ${E}/{\Q_3}$ be an elliptic curve with good reduction such that $a_3=a_3(\Et)\in\{-3,0,3\}$ and $\D=\Dcris[\Q_3](V_3(E))$.
There exists an isomorphism of filtered $\phi$-modules between $\D$ and $\Dc(1;a_3;0)$. Moreover, if $a,b\in\{ -3,0,3\}$
then $\Dc(1;a;0)$ and $\Dc(1;b;0)$ are isomorphic if and only if $a=b$.
\end{proposition}
\begin{proof}
Let $D$ (resp. $D^{\prime}$) be the $\Q_3$-vector space associated to $\D$ (resp. $\Dc(1;a_3;0)$).
Let $B=(e_1,e_2)$ and $B^{\prime}=(e_1^{\prime},e_2^{\prime})$ be basis for $D$ and $D^{\prime}$ respectively such that
\[
M_B(\phi)=\begin{pmatrix} 0 & -3 \\ 1 & -a_3 \end{pmatrix}=M_{B^{\prime}}(\phi^{\prime}).
\]
Such a basis of $D$ always exists since $P_{\Char}(\phi)(X)=X^2+a_3X+3$ as $\D$ satisfies the condition $(1)$.
A $\Q_3$-isomorphism $\eta$ between $D$ and $D^{\prime}$ is $\phi$-equivariant if and only if
\[
M_{B,B^{\prime}}(\eta)\in C(M_B(\phi)).
\]
Where $C(M_B(\phi))$ denotes the centralizer of $M_B(\phi)$ in $\mathrm{GL}_2(\Q_3)$.
Notice that since
\[
C(M_B(\phi))=C(\Q_3[M_B(\phi)])
\]
and $P_{\Char}(\phi)(X)$ is irreducible, the Double Centralizer Theorem implies
\[
C(M_B(\phi))=\Q_3[M_B(\phi)]=\Q_3(M_B(\phi)).
\]
In particular, every non zero element of $\Q_3(M_B(\phi))$ is an isomorphism of $\phi$-modules between $(D,\phi)$
and $(D^{\prime},\phi^{\prime})$. Consider $\Fil^1 D=(\alpha e_1 + \beta e_2)\Q_3$, $(\alpha,\beta)\neq (0,0)$.
The matrix
\[
\begin{pmatrix} \alpha & -3\beta \\ \beta & \alpha-a_3\beta \end{pmatrix}
\]
is invertible because the homogenous polynomial $X^2-a_3XY+3Y^2$ only has trivial roots in ${(\Q_3)}^2$.
Let $(\lambda,\mu)\in\Q_3^2$ be the unique solution to the system of equations
\[
\begin{pmatrix} \alpha & -3\beta \\ \beta & \alpha-a_3\beta \end{pmatrix}\begin{pmatrix} x \\ y \end{pmatrix} = \begin{pmatrix}
0 \\ 1
\end{pmatrix}
\]
it follows that $(\lambda,\mu)\neq (0,0)$ and
\[
\begin{pmatrix} \lambda & -3\mu \\ \mu & \lambda-a_3\mu \end{pmatrix}\begin{pmatrix} \alpha \\ \beta \end{pmatrix} = \begin{pmatrix}
0 \\ 1
\end{pmatrix}.
\]
Therefore, the map $\lambda\Id+\mu M_B(\phi)\in \Q_3[M_B(\phi)]$ defines an isomorphism of filtered $\phi$-modules
between $\D$ and $\Dc(1;a_3;0)$. One checks the last assertion by a simple computation.
\end{proof}

\begin{remark}
There are $3$ isomorphism classes of filtered $\phi$-modules in the supersingular case, one for each value taken by $a$.
\end{remark}

\subsubsection{The ordinary case}
Let $a\in\{-2,-1,1,2\}$ and $\alpha\in\Proj^1(\Q_3)$. We denote by $\Dc(1;a;\alpha)$ the filtered $\phi$-module defined by:
\begin{itemize}
	\item $D=\Q_3 e_1\oplus\Q_3 e_2$
	\item $M_B(\phi)=\begin{pmatrix} u & 0 \\ 0 & u^{-1}3 \end{pmatrix}$, where $u\in\Z_3^{\times}$
    such that $u+u^{-1}3=-a$
	\item $\Fil^1D=(\alpha e_1 +e_2)\Q_3$.
\end{itemize}
For each $a\in\{-2,-1,1,2\}$ and each $\alpha\in\Proj^1(\Q_3)$, the filtered $\phi$-module $\Dc(1;a;\alpha)$ satisfies conditions $(1)-(4)$ and is admissible for $\alpha\neq\infty$.

\begin{proposition}
Let ${E}/{\Q_3}$ be an elliptic curve with good reduction such that $a_3=a_3(\Et)\in\{-2,-1,1,2\}$ and $\D=\Dcris[\Q_3](V_3(E))$.
There exists an isomorphism of filtered $\phi$-modules between $\D$ and $\Dc(1;a_3;\alpha)$ for some $\alpha\in\{0,1\}$.
Moreover, if $(\alpha,a),(\beta,b)\in\{ 0,1\}\times \{ -2,-1,1,2\}$ then $\Dc(1;a;\alpha)$ and $\Dc(1;b;\beta)$ are isomorphic
if and only if $(\alpha,a)=(\beta,b)$.
\end{proposition}
\begin{proof}
Since $\D$ is admissible, the only possible filtrations are defined by a $\Q_3$-line of the
form $\Fil^1 D=(\beta e_1+e_2)\Q_3$ for some $\beta\in\Q_3$.
Let $\alpha\in\{ 0,1\}$ and $D^{\prime}$ be the $\Q_3$-vector space associated to $\Dc(1;a_3;\alpha)$.
Let $B=(e_1,e_2)$, $B^{\prime}=(e_1^{\prime},e_2^{\prime})$ be basis of $D$ and $D^{\prime}$ respectively, such that
\[
M_B(\phi)=\begin{pmatrix} u & 0 \\ 0 & u^{-1}3 \end{pmatrix}=M_{B^{\prime}}(\phi^{\prime}),\ u\in\Z_3^{\times},\ u+u^{-1}3=-a_3.
\]
Such a basis exists because $\D$ satisfies $(1)$ and $(a_3,3)=1$ and thus we have
\[
P_{\Char}(\phi)(X)=X^2+a_3X+3=(X-u)(X-u^{-1}3)\text{ for some }u\in\Z_3^{\times}.
\]
A $\Q_3$-isomorphism $\eta$ between $D$ and $D^{\prime}$ is $\phi$-equivariant if and only if
\[
M_{B,B^{\prime}}(\eta)\in C(M_B(\phi)).
\]
This time, since $P_{\Char}(\phi)(X)$ is a product of distinct linear factors
\[
C(\Q_3[M_B(\phi)])=\Q_3[M_B(\phi)]=\left\lbrace \begin{pmatrix}\lambda & 0 \\ 0 & \mu\end{pmatrix}:\ \lambda,\mu\in\Q_3\right\rbrace\simeq\Q_3[X]/(X-u)\times\Q_3[X]/(X-u^{-1}3).
\]
If $\beta=0$, then every invertible element of $C(M_B(\phi))$ defines an isomorphism of filtered $\phi$-modules
between $\D$ and $\Dc(1;a_3;0)$. If $\beta\neq 0$, then taking $\lambda=\beta^{-1}$ and $\mu=1$ gives an isomorphism
of filtered  $\phi$-modules between $\D$ and $\Dc(1;a_3;1)$.
\end{proof}

\begin{remark}
There are $8$ isomorphism classes of filtered $\phi$-modules in the ordinary case, two for each possible value taken by $a$.
\end{remark}

\begin{remark}
The elliptic curves $E/\Q_3$ with ordinary good reduction and $\alpha=0$ are canonical lifts of their corresponding reduced curve $\Et/\Fd_3$.
\end{remark}

\subsection{\for{toc}{The quadratic case}\except{toc}{The quadratic case ($e=2$)}}

Let ${E}/{\Q_3}$ with semi-stability defect $\dst(E)=2$. Since $2$ and $3$ are coprime, the only quadratic extension
of $\Q_3^{\mathrm{un}}$ is $\Q_3^{\mathrm{un}}(\sqrt{3})$. Let $K=\Q_3(\sqrt{3})$, it is a Galois extension of
degree $2$ with Galois group $\langle \tau_2\rangle$ over which $E$ acquire good reduction. As usual we denote
by $a_3=a_3(\Et)$ the trace of the Frobenius of ${\Et}/{\Fd_3}$.

\subsubsection{The supersingular case}

Let $a\in\{ -3,0,3\}$ and $\alpha\in\Proj^1(K)$. We denote by $\Dc(2;a;\alpha)$ the filtered $(\phi,\Gal({K}/{\Q_3}))$-module
defined by:
\begin{itemize}
	\item $D=\Q_3 e_1\oplus\Q_3 e_2$
	\item $M_B(\phi)=\begin{pmatrix} 0 & -3 \\ 1 & -a\end{pmatrix}$
	\item $M_B(\tau_2)=\begin{pmatrix} -1 & 0 \\ 0 & -1\end{pmatrix}$
	\item $\Fil^1D_K=(\alpha\otimes e_1 + 1\otimes e_2)K$, where $D_K=K\otimes_{K_0}D$.
\end{itemize}
For each $a\in\{-3,0,3\}$ and each $\alpha\in\Proj^1(K)$, the filtered $(\phi,\Gal({K}/{\Q_3}))$-module $\Dc(2;a;\alpha)$ satisfies conditions $(1)-(4)$ and is admissible.

\begin{proposition}
Let ${E}/{\Q_3}$ be an elliptic curve with $\dst(E)=2$ such that $a_3=a_3(\Et)\in\{-3,0,3\}$ and $\D=\Dcris[K](V_3(E))$.
There exists an isomorphism of filtered $(\phi,\Gal({K}/{\Q_3}))$-modules between $\D$ and $\Dc(2;a_3;0)$. Moreover,
if $a,b\in\{ -3,0,3\}$ then $\Dc(2;a;0)$ and $\Dc(2;b;0)$ are isomorphic if and only if $a=b$.
\end{proposition}
\begin{proof}
Let $D$ be the underlying $\Q_3$-vector space associated to $\D$ and $B=(e_1,e_2)$ a basis of $D$ such that
\[
M_B(\phi)=\begin{pmatrix} 0 & -3 \\ 1 & -a_3\end{pmatrix}.
\]
The element $\tau_2$ is seen as a $\Q_3$-automorphism of $D$ and is of order $2$. Since $\D$ satisfies conditions $(2)-(3)$,
we have $P_{\Char}(\tau_2)(X)\in\Q[X]$ and $\detm(\tau_2)=1$, so that
\[
P_{\Char}(\tau_2)(X)={(X+1)}^2
\]
thus $\tau_2=-\Id$. The $K$-line $(1\otimes e_1)K$ is stable by $\tau_2$ and if $\alpha\in K$, the $K$-line
$(\alpha\otimes e_1 +1\otimes e_2)K$ is stable by $\tau_2$ if and only if $\alpha\in\Q_3$. Let $\alpha\in\Proj^1(K)$ such
that $\Fil^1 D_K=(\alpha\otimes e_1 +1\otimes e_2)K$ is the $K$-line defining the filtration of $\D$, we will
show that $\D\simeq \Dc(2;a_3;0)$. If $\alpha=0$ it is obvious. If $\alpha=\infty$, the isomorphism is given
by $e_1\mapsto e_2$ and $e_2\mapsto -3e_1$. Finally, if $\alpha\neq 0,\infty$ it is given by $e_1\mapsto (3/\alpha)e_1 + e_2$
and $e_2\mapsto -3e_1 + (3/\alpha-a_3)e_2$.
\end{proof}

\subsubsection{The ordinary case}

Let $a\in\{-2,-1,1,2\}$ and $\alpha\in\Proj^1(\Q_3)$. We denote by $\Dc(2;a;\alpha)$ the filtered $(\phi,\Gal({K}/{\Q_3}))$-module
defined by:
\begin{itemize}
	\item $D=\Q_3 e_1\oplus \Q_3 e_2$,
	\item $M_B(\phi)=\begin{pmatrix} u & 0\\ 0 & u^{-1}3\end{pmatrix}$ where $u\in\Z_3^{\times}$ such that $u+u^{-1}3=-a$
	\item $M_B(\tau_2)=\begin{pmatrix} -1 & 0 \\ 0 & -1\end{pmatrix}$
	\item $\Fil^1D_K=(\alpha\otimes e_1+1\otimes e_2)K$.
\end{itemize}
For each $a\in\{-2,-1,1,2\}$ and each $\alpha\in\Proj^1(\Q_3)$, the filtered $(\phi,\Gal({K}/{\Q_3}))$-module $\Dc(2;a;\alpha)$ satisfies conditions $(1)-(4)$ and is admissible for $\alpha\neq\infty$.

\begin{proposition}
Let ${E}/{\Q_3}$ be an elliptic curve with $\dst(E)=2$ such that $a_3=a_3(\Et)\in\{-2,-1,1,2\}$ and $\D=\Dcris[K](V_3(E))$.
There exists an isomorphism of filtered $(\phi,\Gal({K}/{\Q_3}))$-modules between $\D$ and $\Dc(2;a_3;\alpha)$ for
some $\alpha\in\{ 0,1\}$. Moreover, if $(\alpha,a),(\beta,b)\in\{0,1\}\times\{-2,-1,1,2\}$ then $\Dc(2;a;\alpha)$
and $\Dc(2;b;\beta)$ are isomorphic if and only if $(\alpha,a)=(\beta,b)$.
\end{proposition}
\begin{proof}
See the ordinary crystalline case for the description of $\phi$ and the supersingular quadratic case for the description
of $\tau_2$ and the filtration.
\end{proof}

\begin{remark}
These are exactly the twists by the ramified quadratic character associated to ${\Q_3(\sqrt{3})}/{\Q_3}$ of the corresponding crystalline cases.
\end{remark}

\begin{remark}
    As in the crystalline case, the elliptic curves $E/\Q_3$ with ordinary potential good reduction and $\alpha=0$ are canonical lifts of their corresponding reduced curve $\Et/\Fd_3$.
\end{remark}

\subsection{\for{toc}{The quartic case}\except{toc}{The quartic case ($e=4$)}}

Let ${E}/{\Q_3}$ with semi-stability defect $\dst(E)=4$. Again, since $4$ and $3$ are coprime, the only quartic extension
of $\Q_3^{\mathrm{un}}$ is $\Q_3^{\mathrm{un}}(\sqrt[4]{3})$. Let us fix $\zeta_4$ a primitive fourth root of unity and $\pi_4$ a root of $X^4-3$ in $\Qal_3$.
Consider $L=\Q_3(\pi_4)$, $K=L(\zeta_4)$ its algebraic closure and $K_0=\Q_3(\zeta_4)$ the maximal unramified
extension of ${K}/{\Q_3}$. Our curve necessarily acquires good reduction over $L$. Let $\sigma\in G({K_0}/{\Q_3})$ be the absolute
Frobenius on $K_0$, $\omega\in G({K}/{\Q_3})$ a lifting of $\sigma$ fixing $L$ and $\tau_4$ a generator of $G(K/K_0)=I({K}/{\Q_3})$.
Then $G({K}/{\Q_3})=\langle \tau_4\rangle\rtimes \langle \omega \rangle$ with $\tau_4\omega =\omega\tau_4^{-1}$.

\vspace{\baselineskip}

Let $\alpha\in\Proj^1(\Q_3)$. We denote by $\Dpc(4;0;\alpha)$ the filtered $(\phi,\Gal({K}/{\Q_3}))$-module defined by:
\begin{itemize}
	\item $D=K_0e_1\oplus K_0e_2$
	\item $\phi (e_1)=e_2,\ \phi (e_2)=-3e_1$
	\item $M_B(\tau_4)=\begin{pmatrix} \zeta_4 & 0 \\ 0 & \zeta_4^{-1}\end{pmatrix}$
	\item $\omega (e_1)=e_1,\ \omega (e_2)=e_2$
	\item $\Fil^1D_K=(\alpha\pi_4^{-1}\otimes e_1 + \pi_4\otimes e_2)K$.
\end{itemize}
For each $\alpha\in\Proj^1(\Q_3)$, the filtered $(\phi,\Gal({K}/{\Q_3}))$-module $\Dpc(4;0;\alpha)$ satisfies conditions $(1)-(4)$ and is admissible.

\begin{proposition}
Let ${E}/{\Q_3}$ be an elliptic curve with $\dst(E)=4$ and $\D=\Dcris[K](V_3(E))$. There exists an isomorphism of
filtered $(\phi,\Gal({K}/{\Q_3}))$-modules between $\D$ and $\Dpc(4;0;\alpha)$. Moreover if $\alpha,\beta\in\Proj^1(\Q_3)$,
then $\Dpc(4;0;\alpha)\simeq \Dpc(4;0;\beta)$ if and only if $\alpha =\beta$.
\end{proposition}
\begin{proof}
Let $D$ be the underlying $K_0$-vector space associated to $\D$, the element $\tau_4$ acts $K_0$-linearly over $D$ and
the morphism
\[
I({K}/{\Q_3})\longrightarrow\Aut_{K_0}(D)
\]
is injective by minimality of $e({K}/{\Q_3})$. We identify $\tau_4$ to its image in $\Aut_{K_0}(D)$, it is an element of order $4$.
Again, because $\D$ satisfies $(2)-(3)$ we have $\detm(\tau_4)=1$ and $P_{\Char}(\tau_4)(X)\in\Q[X]$ so that
\[
P_{\Char}(\tau_4)(X)=P_{\min}(\tau_4)(X)=X^2+1=(X-\zeta_4)(X-\zeta_4^{-1})
\]
in particular $\tau_4$ is diagonalizable in $K_0$ and has distinct eigenvalues. Let $(e_1,e_2)$ be a diagonalization basis
of $\tau_4$ over $K_0$. The relation $\tau_4\omega =\omega\tau_4^{-1}$ implies that $\omega(e_i)\in K_0e_i$, $i=1,2$.
Denote by $\omega_i=\omega\vert_{K_0e_i}$, the group $\langle\omega_i\rangle$ acts semi-linearly over $K_0e_i$.
Descent theory tells us that ${(K_0e_i)}^{\langle\omega_i\rangle}\neq\{0\}$. We can then find a basis $(e_1,e_2)$ of $D$
over $K_0$ which is fixed by $\omega$ and such that $\tau_4 (e_1)=\zeta_4e_1,\ \tau_4 (e_2)=\zeta_4^{-1}e_2$. Since $\phi$
is $\Gal({K}/{\Q_3})$-equivariant, it commutes to $\tau_4$ and $\omega$, a simple calculation shows that $\phi (e_1)\in\Q_3 e_2$
and $\phi (e_2)\in\Q_3 e_1$. Since $\detm(\phi)=3$, $\phi (e_1)=ae_2$ and $\phi (e_2)=-3a^{-1}e_1$,
necessarily $a\in\Q_3^{\times}$. That way we show that there exists a $K_0$-basis of $D$ such that
\[
\phi (e_1)=e_2,\ \phi (e_2)=-3e_1,\ \tau_4 (e_1)=\zeta_4e_1,\ \tau_4 (e_2)=\zeta_4^{-1}e_2,\ \omega (e_1)=e_1,\ \omega (e_2)=e_2
\]
We have now described the $(\phi,\Gal({K}/{\Q_3}))$-module structure on $D$. In particular, we see that $a_3=0$, i.e. $\Et_L$ is
supersingular, but the two other supersingular values $3$ and $-3$ cannot appear.

What is left is to determine the $K$-line $\Fil^1D_K$ which defines the filtration, it needs to satisfy the weak admissibility
condition and be $\Gal({K}/{\Q_3})$-stable. Since $\phi$ does not have any stable subspaces, it is immediate.
The $K$-line $(1\otimes e_1)K$ is stable by action of $\Gal({K}/{\Q_3})$. Let $\beta\in\Q_3$
and $\Fil^1D_K=(\beta\otimes e_1+1\otimes e_2)K$. One easily shows that $\Fil^1 D_K$ is stable by $\omega$ if and
only if $\beta\in L$ and by $\tau_4$ if and only if $\pi_4^2\beta\in K_0$. Then $\Fil^1 D_K$ is stable by
action of $\Gal({K}/{\Q_3})$ if and only if $\pi_4^2\beta\in L\cap K_0=\Q_3$. Let $\alpha=\pi_4^2\beta\in\Q_3$,
we can rewrite our $K$-line defining the filtration as
\[
\Fil^1 D_K=(\alpha \pi_4^{-1}\otimes e_1 + \pi_4\otimes e_2)K
\]
it is then clear that $\D\simeq\Dpc(4;0;\alpha)$.

Let $\alpha,\beta\in\Proj^1(\Q_3)$, consider the following filtered $(\phi,\Gal({K}/{\Q_3}))$-modules:
$\D=\Dpc(4;0;\alpha)$, $\D^{\prime}=\Dpc(4;0;\beta)$ and let $B=(e_1,e_2)$, $B^{\prime}=(e_1^{\prime},e_2^{\prime})$
be $K_0$-basis of $D$ and $D^{\prime}$ their respective underlying $K_0$-vector spaces. Let $\psi:\D\longrightarrow \D^{\prime}$
be a non zero morphism of filtered $(\phi,\Gal({K}/{\Q_3}))$-modules. Let $D_0=D^{\langle\omega\rangle}$
and $D_0^{\prime}={(D^{\prime})}^{\langle\omega^{\prime}\rangle}$. The relation $\psi\circ\omega = \omega^{\prime}\circ\psi$
implies $\psi(D_0)\subseteq D_0^{\prime}=\Q_3 e_1^{\prime}\oplus\Q_3 e_2^{\prime}$.
Moreover, $\psi\circ\tau_4 =\tau_4^{\prime}\circ\psi$ implies $\psi (e_i)\in K_0 e_i^{\prime}$, $i=1,2$.
Then there exists $a,d\in\Q_3$ such that $\psi (e_1)=ae_1^{\prime}$ and $\psi (e_2)=de_2^{\prime}$.
Finally, $\psi\circ\phi =\phi^{\prime}\circ\psi$ leads to $a=d$. Denoting by $\psi_K$ the $K$-linear extension
of $\psi$, we see that $\psi_K(\Fil^1D_K)\subseteq\Fil^1D_K^{\prime}$ if and only if $\alpha=\beta$.
\end{proof}

\subsection{\for{toc}{The cubic case}\except{toc}{The cubic case ($e=3$)}}

Let ${E}/{\Q_3}$ be an elliptic curve with semi-stability defect $\dst(E)=3$. There are exactly $9$ totally ramified extensions
of degree $3$ of $\Q_3$ (see~\cite{LMFDB}). Since $E$ acquires good reduction over a degree $3$ Galois extension
of $\Q_3^{\mathrm{un}}$, we are interested in the ones that keep their ramification index after Galois closure.
Indeed, if $e({F}/{\Q_3})=3$ but $e({F^{\Gal}}/{\Q_3})>3$, then $[{(F^{\Gal})}^{\mathrm{un}}:\Q_3^{\mathrm{un}}]>3$ is not minimal.
There are only $4$ such extensions; among these, $3$ are Abelian and the last one has a Galois closure of degree $6$ with
Galois group isomorphic to $S_3$. One easily shows (using the Kronecker-Weber Theorem) that the three considered Abelian
extensions are exactly the degree $3$ totally ramified subextensions of $\Q_3(\zeta_{13},\zeta_9+\zeta_9^{-1})$, so their compositum with $\Q_3^{\mathrm{un}}$ is $\Q_3^{\mathrm{un}}(\zeta_9+\zeta_9^{-1})$, and they are therefore interchangeable. This is not the case of the non Abelian extension. Let $\Lg=\Q_3(\zeta_9+\zeta_9^{-1})$ (resp.\ $\Lng=\Q_3(X^3-3X^2+6)$) be a minimal field of good reduction for ${E}/{\Q_3}$ in the Abelian (resp.\ non Abelian) case.

Given an elliptic curve $E/\Q_3$ and $\ell\neq 3$ a prime, we consider:
\[
\tau_E: I({\Qal_3}/{\Q_3})\overset{\rho_{E,\ell}}{\longrightarrow}\mathrm{GL}_2(\Q_{\ell})\hookrightarrow\mathrm{GL}_2(\mathbb{C}).
\]

\begin{proposition}
Let $E,{E^{\prime}}/{\Q_3}$ be elliptic curves with $\dst(E)=\dst(E^{\prime})=3$. We have the following equivalence:
\[
\tau_E\simeq_{\mathbb{C}}\tau_{E^{\prime}}\Leftrightarrow M_E=M_{E^{\prime}}.
\]
\end{proposition}
\begin{proof}
The left to right implication is obvious
since $M_E={(\Q_3^{\mathrm{un}})}^{\ker(\tau_E)}$, $M_{E^{\prime}}={(\Q_3^{\mathrm{un}})}^{\ker(\tau_{E^{\prime}})}$ and two
isomorphic representations share the same kernel. If $M_E=M_E^{\prime}$ then $\ker(\tau_E)=\ker(\tau_{E^{\prime}})$ and
both types factors into faithful irreducible representations of $\Gal(M_{E}/\Q_3^{\mathrm{un}})\simeq\Z/3\Z$ defined over $\Q$,
which are necessarily isomorphic.
\end{proof}

Using~\cite{DFV},Table 1 we see that there are only two isomorphism classes of such objetcs for $p=e=3$ so
that $\Lg\Q_3^{\mathrm{un}}$ and $\Lng\Q_3^{\mathrm{un}}$ are indeed distinct.

\subsubsection{The non Abelian case}
Let ${E}/{\Q_3}$ with $\dst(E)=3$ acquiring good reduction over $\Lng$ with $K=\Lng(\zeta_4)$ its Galois closure.
Denote by $K_0$ the maximal unramified extension of ${K}/{\Q_3}$ and $\sigma\in \Gal({K_0}/{\Q_3})$ the absolute Frobenius.
Let $\omega\in \Gal({K}/{\Q_3})$ be a lifting of $\sigma$ fixing $\Lng$ and $\tau_3$ a generator of $\Gal(K/K_0)=I({K}/{\Q_3})$.
Then,  $\Gal({K}/{\Q_3})=\langle \tau_3\rangle\rtimes \langle \omega \rangle$ with $\tau_3\omega =\omega\tau_3^{-1}$
(the unique non trivial semi-direct product).

Let $\alpha\in\mathcal{M}_3^{\mathrm{na}}:=\{\alpha\in \Lng\vert \tau_3(\alpha)=(3\zeta_4+\alpha)/(1+\zeta_4\alpha)\}$.
We denote by $\Dpcng(3;0;\alpha)$ the filtered $(\phi,\Gal({K}/{\Q_3}))$-module defined by:
\begin{itemize}
	\item $D=K_0e_1\oplus K_0e_2$
	\item $\phi (e_1)=e_2,\ \phi (e_2)=-3e_1$
    \item $M_B(\tau_3)=\begin{pmatrix} -\frac{1}{2} & \frac{3}{2}\zeta_4 \\ \frac{1}{2}\zeta_4 & -\frac{1}{2}\end{pmatrix}$
	\item $\omega (e_1)=e_1,\ \omega (e_2)=e_2$
	\item $\Fil^1D_K=(\alpha\otimes e_1 + 1\otimes e_2)K$.
\end{itemize}
For each $\alpha\in\mathcal{M}_3^{\mathrm{na}}$, the filtered $(\phi,\Gal({K}/{\Q_3}))$-module $\Dpcng(3;0;\alpha)$ satisfies conditions $(1)-(4)$ and is admissible.

\begin{proposition}
Let ${E}/{\Q_3}$ be an elliptic curve with $\dst(E)=3$ acquiring good reduction over $\Lng$ and $\D=\Dcris[K](V_3(E))$.
There exists $\alpha\in\mathcal{M}_3^{\mathrm{na}}$ such that $\D$ and $\Dpcng(3;0;\alpha)$ are isomorphic as
filtered $(\phi,\Gal({K}/{\Q_3}))$-modules. Moreover, if $\alpha,\beta\in\mathcal{M}_3^{\mathrm{na}}$,
then $\Dpcng(3;0;\alpha)\simeq \Dpcng(3;0;\beta)$ if and only if $\alpha =\beta$.
\end{proposition}
\begin{proof}
Denote by $D$ the underlying $K_0$-vector space associated to $\D$, the element $\tau_3$ acts $K_0$-linearly over $D$ and
the morphism
\[
I({K}/{\Q_3})\longrightarrow\Aut_{K_0}(D)
\]
is injective by minimality of $e({K}/{\Q_3})$. We identify $\tau_3$ to its image in $\Aut_{K_0}(D)$, it is an element of order $3$.
Since $\zeta_3\notin K_0$,
\[
P_{\Char}(\tau_3)(X)=P_{\min}(\tau_3)(X)=X^2+X+1.
\]
Let $B=(e_1,e_2)$ be a $K_0$-basis of $D$ fixed by $\omega$ such that $\phi(e_1)=e_2$ and $\phi(e_2)=-3e_1-a_3e_2$.
Such a basis always exists since $\omega$ acts semi-linearly over $D$ and $\phi\omega=\omega\phi$.
Let $\lambda,\mu,\lambda^{\prime},\mu^{\prime}\in K_0$ such that
\[
M_B(\tau_3)=\begin{pmatrix} \lambda & \mu^{\prime} \\ \mu & \lambda^{\prime}\end{pmatrix}.
\]
We already know that $\lambda^{\prime}=-\lambda-1$ and $\mu^{\prime}=P(\lambda)(-\mu)$ where $P=P_{\Char}(\tau_3)$.
The relations $\tau\omega=\omega\tau^{-1}$ and $\tau\phi=\phi\tau$ imply that $a_3=0$,
$\sigma(\lambda)=-\lambda-1$, $\sigma(\mu)=-\mu$ and $P(\lambda)/(-\mu)=3\mu$.

In conclusion:
\begin{itemize}
	\item $\phi(e_1)=e_2,\ \phi(e_2)=-3e_1$
	\item $\omega(e_1)=e_1,\ \omega(e_2)=e_2$
	\item $M_B(\tau_3)=\begin{pmatrix} \lambda & 3\mu \\ \mu & -\lambda-1\end{pmatrix}$,$\ \lambda\in -\frac{1}{2}+\Q_3\zeta_4,\ \mu\in\Q_3^{\times}\zeta_4,\text{ and } P(\lambda)+3\mu^2=0$.
\end{itemize}
Let
\[
M=\begin{pmatrix} \lambda +\frac{1}{2} & -3\mu -\frac{3}{2}\zeta_4 \\ \mu -\frac{1}{2}\zeta_4 & \lambda +\frac{1}{2}\end{pmatrix}
\]
clearly $\det(M)=0$, let $(a,b)\in{\ker(M)}^{G({K}/{\Q_3})}\subseteq K_0^2$ be a non zero element. Then
\[
B^{\prime}=(e_1^{\prime},e_2^{\prime})=(ae_1+be_2, -3be_1+ae_2)
\]
is a $K_0$-basis of $D$ such that
\begin{itemize}
	\item $\phi(e_1^{\prime})=e_2^{\prime},\ \phi(e_2^{\prime})=-3e_1^{\prime}$
	\item $\omega(e_1^{\prime})=e_1^{\prime},\ \omega(e_2^{\prime})=e_2^{\prime}$
	\item $M_{B^{\prime}}(\tau_3)=\begin{pmatrix} -\frac{1}{2} & \frac{3}{2}\zeta_4 \\ \frac{1}{2}\zeta_4 & -\frac{1}{2}\end{pmatrix}$.
\end{itemize}
Again, we denote by $(e_1,e_2)$ such a basis.
One easily checks that $(1\otimes e_1)K$ and $(1\otimes e_2)K$ are not stable by $\tau_3$.
Let $\alpha\in K^{\times}$ and $\Fil^1 D_K =(\alpha\otimes e_1 + 1\otimes e_2)$.
A simple calculation shows that such a $K$-line is stable by the action of $\Gal({K}/{\Q_3})$ if and only if $\alpha\in\Lng$ and
$\tau_3(\alpha)=(3\zeta_4+\alpha)/(1+\zeta_4\alpha)$.

Let $B,B^{\prime}$ be $K_0$-basis of $\D=\Dpcng(3;0;\alpha)$ and $\D^{\prime}=\Dpcng(3;0;\beta)$ respectively.
One easily shows that an isomorphism $\eta$ of $(\phi,\Gal({K}/{\Q_3}))$-modules between $\D$ and $\D^{\prime}$ is of the form
\[
M_{B,B^{\prime}}(\eta)=\begin{pmatrix} a & 0 \\ 0 & a\end{pmatrix},\ a\in\Q_3^{\times}.
\]
Denoting by $\eta_K:D_K\longrightarrow D_K^{\prime}$ the $K$-linear extension of $\eta$, it is then clear that
\[
\eta_K((\alpha\otimes e_1 +1\otimes e_2)K)\subseteq (\beta\otimes e_1 +1\otimes e_2)K \quad\Leftrightarrow\quad\alpha=\beta .
\]
\end{proof}

\subsubsection{The Abelian case}
Let ${E}/{\Q_3}$ with $\dst(E)$ acquiring good reduction over a degree $3$ Abelian extension of $\Q_3$.
There are only $4$ such extensions and among them the unique unramified one. These are exactly the sub-extensions
of $\Q_3(\zeta_{13},\zeta_9+\zeta_9^{-1})$, hence they share the same maximal unramified extension inside $\Qal_3$.
Let $K=\Lg$ be one of these $3$ extensions, its Galois group $\Gal({K}/{\Q_3})=I({K}/{\Q_3})=\langle \tau_3\rangle$ is cyclic of order $3$.

Let $\alpha\in\mathcal{M}_3^{\mathrm{a}}:=\{ \alpha\in \Lg\vert\tau_3(\alpha)=(\alpha -1)/\alpha\}$ and
\[
    (a,\mu)\in(\{-3\}\times\{1,2\})\sqcup (\{0\}\times\{-1,1\})\sqcup (\{3\}\times\{-2,-1\}).
\]
We denote by $\Dpcg(3;a,\mu;\alpha)$ the filtered $(\phi,\Gal({K}/{\Q_3}))$-module defined by:
\begin{itemize}
	\item $D=\Q_3 e_1\oplus \Q_3 e_2$,
	\item $M_B(\tau_3)=\begin{pmatrix} 0 & -1 \\ 1 & -1\end{pmatrix}$
    \item $a=-3$: $\mu=1$: $M_B(\phi)=\begin{pmatrix} 1 & 1 \\ -1 & 2\end{pmatrix}$,$\ $ $\mu=2$: $M_B(\phi)=\begin{pmatrix} 2 & -1 \\ 1 & 1\end{pmatrix}$
	\item $a=0$: $\mu=1$: $M_B(\phi)=\begin{pmatrix} 1 & -2 \\ 2 & -1\end{pmatrix}$,$\ $ $\mu=-1$: $M_B(\phi)=\begin{pmatrix} -1 & 2 \\ -2 & 1\end{pmatrix}$
	\item $a=3$: $\mu=-1$: $M_B(\phi)=\begin{pmatrix} -1 & -1 \\ 1 & -2\end{pmatrix}$,$\ $ $\mu=-2$: $M_B(\phi)=\begin{pmatrix} -2 & 1 \\ -1 & -1\end{pmatrix}$
	\item $\Fil^1D_K=(\alpha\otimes e_1+1\otimes e_2)K$.
\end{itemize}
For each $\alpha\in\mathcal{M}_3^{\mathrm{a}}$, the filtered $(\phi,\Gal({K}/{\Q_3}))$-module $\Dpcg(3;a,\mu;\alpha)$ satisfies conditions $(1)-(4)$ and is admissible.

\begin{proposition}
Let ${E}/{\Q_3}$ be an elliptic curve with $\dst(E)=3$ acquiring good reduction over $K$ and $\D=\Dcris[K](V_3(E))$.
There exists $\alpha\in\mathcal{M}_3^{\mathrm{a}}$
and $(a,\mu)\in(\{-3\}\times\{1,2\})\sqcup (\{0\}\times\{-1,1\})\sqcup (\{3\}\times\{-2,-1\})$ such that $\D$
and $\Dpcg(3;a,\mu;\alpha)$ are isomorphic as filtered $(\phi,\Gal({K}/{\Q_3}))$-modules.

Moreover, if $\alpha,\beta\in\mathcal{M}_3^{\mathrm{a}}$
and $(a,\mu),(b,\nu)\in(\{-3\}\times\{1,2\})\sqcup(\{0\}\times\{-1,1\})\sqcup(\{3\}\times\{-2,-1\})$,
then $\Dpcg(3;a,\mu;\alpha)\simeq \Dpcg(3;b,\nu;\beta)$ if and only if $(a,\mu) =(b,\nu)$.
\end{proposition}
\begin{proof}
Denote by $D$ the underlying $\Q_3$-vector space associated to $\D$, the element $\tau_3$ acts $\Q_3$-linearly over $D$ and
the natural morphism
\[
I({K}/{\Q_3})\longrightarrow\Aut_{\Q_3}(D)
\]
is injective by minimality of $e({K}/{\Q_3})$. We identify $\tau_3$ to its image in $\Aut_{\Q_3}(D)$, it is an element of order $3$.
Since $\zeta_3\notin \Q_3$,
\[
P_{\Char}(\tau_3)(X)=P_{\min}(\tau_3)(X)=X^2+X+1.
\]
Let $B=(e_1,e_2)$ be a $\Q_3$-basis of $D$ such that:
\[
M_B(\tau_3)=\begin{pmatrix} 0 & -1 \\ 1 & -1 \end{pmatrix}.
\]
Since $P_{\Char}(\phi)(X)=X^2+a_3X+3$ and $\phi\tau_3=\tau_3\phi$, we have
\[
M_B(\phi)=\begin{pmatrix} \lambda & -2\lambda-a_3 \\ 2\lambda+a_3 & -\lambda-a_3 \end{pmatrix},\ \lambda\in\Q_3
\]
with $\det(\phi)=3\lambda^2+3\lambda a_3+a_3^2=3$, i.e. $\lambda$ is a root of $3X^2+3a_3X+a_3^2-3$. But this polynomial
has roots in $\Q_3$ if and only if $3\mid a_3$, so $a_3\in\{-3,0,3\}$. Considering every possible value of $a_3$ we obtain:
\begin{itemize}
	\item if $a_3=0$, $\lambda$ is a root of $X^2-1$ i.e. $\lambda\in\{-1,1\}$
	\item if $a_3=3$, $\lambda$ is a root of $X^2+3X+2$ i.e. $\lambda\in\{-2,-1\}$
	\item if $a_3=-3$, $\lambda$ is a root of $X^2-3X+2$ i.e. $\lambda\in\{1,2\}$.
\end{itemize}
Observe that $(1\otimes e_1)K$ and $(1\otimes e_2)K$ are not stable by action of $G({K}/{\Q_3})$. Let $\alpha\in K^{\times}$,
the $K$-line $(\alpha\otimes e_1+1\otimes e_2)K$ is stable by $\tau_3$ if and only if $\tau_3(\alpha)=(\alpha -1)/\alpha$.
So that $\D\simeq\Dpcg(3;a_3,\lambda;\alpha)$ with $\alpha$ and $(a_3,\lambda)$ satisfying the desired conditions.

Let $\alpha,\beta\in\mathcal{M}_3^{\mathrm{a}}$
and $(a,\mu),(b,\nu)\in\{-3\}\times\{1,2\}\sqcup\{0\}\times\{-1,1\}\sqcup\{3\}\times\{-2,-1\}$.
Consider $\D=\Dpcg(3;a,\mu;\alpha)$ and $\D^{\prime}=\Dpcg(3;b,\nu;\beta)$, we will first show that their
underlying $(\phi,\Gal({K}/{\Q_3}))$-modules are not isomorphic. Let $B$ and $B^{\prime}$ be $\Q_3$-basis of $D$ and $D^{\prime}$
respectively. A morphism $\eta:D\longrightarrow D^{\prime}$ commuting to $\tau_3$ and $\phi$ must be of the form
\[
M_{B,B^{\prime}}(\eta)=\begin{pmatrix} c & -d \\ d & c-d\end{pmatrix}
\]
where $(c,d)\in\Q_3^2$ is in the kernel of the following linear map
\[
\begin{pmatrix} \lambda -\mu & 2(\mu-\lambda)+b-a \\ \mu-\lambda +b-a & \mu-\lambda\end{pmatrix}.
\]
The determinant of this matrix is $-(3{(\mu-\lambda)}^2+3(\mu-\lambda)(b-a)+{(b-a)}^2)$. There exists $(c,d)\neq (0,0)$ in
the kernel if and only if $(\mu-\lambda)$ is a root of
\[
3X^2+3(b-a)X+{(b-a)}^2.
\]
But such a polynomial has roots in $\Q_3$ if and only if $a=b$, in which case its roots are zero. This shows that if $\D$
and $\D^{\prime}$ are isomorphic as $(\phi,\Gal({K}/{\Q_3}))$-modules then $\lambda=\mu$ and $a=b$. Now suppose
that $(\lambda,a)=(\mu,b)$, let $\Fil^1 D_K=(\alpha\otimes e_1+1\otimes e_2)K$
and $\Fil^1 D^{\prime}_K=(\beta\otimes e_1+1\otimes e_2)K$, the $K$-lines defining the filtrations on $\D$
and $\D^{\prime}$. If $\alpha=\beta$ taking $c=1$ and $d=0$ gives us an obvious isomorphism. In the other case,
we see that $(\alpha\beta -\beta +1)/(\alpha-\beta)\in\Q_3$ and taking some $d\neq 0$
and $c=d(\alpha\beta -\beta +1)/(\alpha-\beta)$ gives us the desired isomorphism.
\end{proof}

\begin{remark}
We observe two differences with the non Abelian case: the supersingular traces $3$ and $-3$ do appear and for each trace value
there are two isomorphism classes of $(\phi,\Gal({K}/{\Q_3}))$-modules (not considering filtration). We will explain the absence
of these traces in the section~\ref{S5}. These two isomorphism classes are unramified quadratic twists of each other.
\end{remark}

\subsection{\for{toc}{The sextic case}\except{toc}{The sextic case ($e=6$)}}

This section can be summarized by the following result: if ${E}/{\Q_3}$ has a semi-stability defect of $3$ then its
quadratic twist ${E^{\prime}}/{\Q_3}$ by the character associated to $\sqrt{3}$ has a semi-stability defect of $6$, and vice versa.
Consequently, if ${F}/{\Q_3}$ is a field of good reduction for $E$, then $F(\sqrt{3})$ is a field of good reduction for $E^{\prime}$. 

\subsubsection{The non Abelian case}
Let ${E}/{\Q_3}$ with $\dst(E)=6$ acquiring good reduction over $\Lng(\sqrt{3})$ and let $K=\Lng(\sqrt{3},\zeta_4)$ be its
Galois closure. We have $\Gal({K}/{\Q_3})=(\langle\tau_3\rangle\times\langle\tau_2\rangle)\rtimes\langle\omega\rangle$
and $I({K}/{\Q_3})=\langle\tau_3\rangle\times\langle\tau_2\rangle$ is cyclic of order $6$.

Let $\alpha\in\mathcal{M}_6^{\mathrm{na}}=\{ \alpha\in \Lng(\sqrt{3})\vert \tau_3(\alpha)=(3\zeta_4+\alpha)(1+\zeta_4\alpha)\}$.
We denote by $\Dpcng(6;0;\alpha)$ the filtered $(\phi,\Gal({K}/{\Q_3}))$-module defined by:
\begin{itemize}
	\item $D=K_0e_1\oplus K_0e_2$
	\item $\phi (e_1)=e_2,\ \phi (e_2)=-3e_1$
	\item $M_B(\tau_3)=\begin{pmatrix} -1/2 & 3/2\zeta_4 \\ 1/2\zeta_4 & -1/2\end{pmatrix}$
	\item $M_B(\tau_2)=\begin{pmatrix} -1 & 0 \\ 0 & -1\end{pmatrix}$
	\item $\omega (e_1)=e_1,\ \omega (e_2)=e_2$
	\item $\Fil^1D_K=(\alpha\otimes e_1 + 1\otimes e_2)K$.
\end{itemize}
For each $\alpha\in\mathcal{M}_6^{\mathrm{na}}$, the filtered $(\phi,\Gal({K}/{\Q_3}))$-module $\Dpcng(6;0;\alpha)$ satisfies
    conditions $(1)-(4)$ and is admissible.

\begin{proposition}
Let ${E}/{\Q_3}$ be an elliptic curve with $\dst(E)=6$ acquiring good reduction over $\Lng(\sqrt{3})$ and $\D=\Dcris[K](V_3(E))$.
There exists $\alpha\in\mathcal{M}_6^{\mathrm{na}}$ such that $\D$ and $\Dpcng(6;0;\alpha)$ are isomorphic as
filtered $(\phi,\Gal({K}/{\Q_3}))$-modules. Moreover if $\alpha,\beta\in\mathcal{M}_6^{\mathrm{na}}$,
then $\Dpcng(6;0;\alpha)\simeq \Dpcng(6;0;\beta)$ if and only if $\alpha =\beta$.
\end{proposition}
\begin{proof}
Similar to the cubic non Abelian case using the natural injection
\[
I({K}/{\Q_3})=\langle\tau_2\rangle\times\langle\tau_3\rangle\hookrightarrow\Aut_{K_0}(D).
\]
\end{proof}

\subsubsection{The Abelian case}
Let ${E}/{\Q_3}$ with $\dst(E)=6$ acquiring good reduction over $K=\Lng(\sqrt{3})$. Then
$\Gal({K}/{\Q_3})=I({K}/{\Q_3})=\langle\tau_3\rangle\times\langle\tau_2\rangle$ is cyclic of order $6$.

Let $\alpha\in\mathcal{M}_6^{\mathrm{a}}=\{ \alpha\in L\vert \tau_3(\alpha)=(\alpha-1)/\alpha\}$
and
\[
    (a,\mu)\in(\{-3\}\times\{1,2\})\sqcup (\{0\}\times\{-1,1\})\sqcup (\{3\}\times\{-2,-1\}).
\]
We denote by $\Dpcg(6;a,\mu;\alpha)$ the filtered $(\phi,\Gal({K}/{\Q_3}))$-module defined by:
\begin{itemize}
	\item $D=\Q_3 e_1\oplus \Q_3 e_2$
	\item $M_B(\tau_3)=\begin{pmatrix} 0 & -1 \\ 1 & -1\end{pmatrix}$
	\item $M_B(\tau_2)=\begin{pmatrix} -1 & 0 \\ 0 & -1\end{pmatrix}$
	\item $a=-3$: $\mu=1$: $M_B(\phi)=\begin{pmatrix} 1 & 1 \\ -1 & 2\end{pmatrix}$,$\ $ $\mu=2$: $M_B(\phi)=\begin{pmatrix} 2 & -1 \\ 1 & 1\end{pmatrix}$
	\item $a=0$: $\mu=1$: $M_B(\phi)=\begin{pmatrix} 1 & -2 \\ 2 & -1\end{pmatrix}$,$\ $ $\mu=-1$: $M_B(\phi)=\begin{pmatrix} -1 & 2 \\ -2 & 1\end{pmatrix}$
	\item $a=3$: $\mu=-1$: $M_B(\phi)=\begin{pmatrix} -1 & -1 \\ 1 & -2\end{pmatrix}$,$\ $ $\mu=-2$: $M_B(\phi)=\begin{pmatrix} -2 & 1 \\ -1 & -1\end{pmatrix}$
	\item $\Fil^1D_K=(\alpha\otimes e_1+1\otimes e_2)K$.
\end{itemize}
For each $\alpha\in\mathcal{M}_6^{\mathrm{a}}$, the filtered $(\phi,\Gal({K}/{\Q_3}))$-module $\Dpcg(6;a,\mu;\alpha)$ satisfies
    conditions $(1)-(4)$ and is admissible.

\begin{proposition}
Let ${E}/{\Q_3}$ be an elliptic curve with $\dst(E)=6$ acquiring good reduction over $K$ and $\D=\Dcris[K](V_3(E))$.
There exists $\alpha\in\mathcal{M}_6^{\mathrm{a}}$
and $(a,\mu)\in(\{-3\}\times\{1,2\})\sqcup (\{0\}\times\{-1,1\})\sqcup (\{3\}\times\{-2,-1\})$ such
that $\D$ and $\Dpcg(6;a,\mu;\alpha)$ are isomorphic as filtered $(\phi,\Gal({K}/{\Q_3}))$-modules.

Moreover if $\alpha,\beta\in\mathcal{M}_6^{\mathrm{a}}$
and $(a,\mu),(b,\nu)\in(\{-3\}\times\{1,2\})\sqcup(\{0\}\times\{-1,1\})\sqcup(\{3\}\times\{-2,-1\})$,
then $\Dpcg(6;a,\mu;\alpha)\simeq \Dpcg(6;b,\nu;\beta)$ if and only if $(a,\mu) =(b,\nu)$.
\end{proposition}
\begin{proof}
Similar to the cubic Abelian case using the following injection
\[
I({K}/{\Q_3})=\langle\tau_2\rangle\times\langle\tau_3\rangle\hookrightarrow\Aut_{\Q_3}(D).
\]
\end{proof}

\subsection{\for{toc}{The dodecic case}\except{toc}{The dodecic case ($e=12$)}}

If an elliptic curve ${E}/{\Q_3}$ has a semi-stability defect $\dst(E)=12$, then its minimal field of good reduction
has Galois closure ${K}/{\Q_3}$ satisfying:

\[
\left\{
    \begin{array}{ll}
        \Gal({K}/{\Q_3})\simeq\Z/3\Z\rtimes D_4\simeq (\Z/3\Z\rtimes \Z/4\Z)\rtimes \Z/2\Z \\
        I({K}/{\Q_3})\simeq\Z/3\Z\rtimes\Z/4\Z.
    \end{array}
\right.
\]
More precisely:
\[
\left\{
    \begin{array}{ll}
        \Gal({K}/{\Q_3})=(\langle \tau_3\rangle\rtimes\langle\tau_4\rangle)\rtimes\langle\omega\rangle \\
        I({K}/{\Q_3})=G(K/K_0)=\langle \tau_3\rangle\rtimes\langle\tau_4\rangle
    \end{array}
\right.
\]
with relations:
\[
\left\{
    \begin{array}{ll}
        \text{ord}(\tau_4)=4, \text{ord}(\tau_3)=3, \text{ord}(\omega)=2\\
        \tau_4\tau_3\tau_4^{-1}=\tau_3^2=\tau_3^{-1}\\
        \omega\tau_4\omega=\tau_4^{-1}\\
        \tau_3\omega=\omega\tau_3 .
    \end{array}
\right.
\]
This follows from the structure of $\Aut_{\Fd_9}(\Et)$, where $\Et$ is the special fibre of $E_K=E\times_{\Q_3}K$. Looking at~\cite{LMFDB} we see that there are exactly $10$ such fields, namely:
\begin{itemize}
    \item $K_1=\Q_3(X^{12} + 3X^4 + 3)$
    \item $K_2=\Q_3(X^{12} - 3X^{11} - 3X^9 + 3X^7 - 3X^4 - 3)$
    \item $K_3=\Q_3(X^{12}+3)$
    \item $K_4=\Q_3(X^{12} + 9X^{10} + 9X^9 - 9X^8 + 6X^6 + 9X^5 - 9X^4 - 3X^3 + 9X^2 - 9X - 12)$
    \item $K_5=\Q_3(X^{12} + 9X^{11} + 9X^{10} + 9X^9 + 9X^8 - 9X^7 - 12X^6 - 9X^2 - 3)$
    \item $K_6=\Q_3(X^{12} + 3X^{10} - 3X^9 - 3X^7 + 3X^6 + 3X^5 + 3X^4 + 3X^3 - 3)$
    \item $K_7=\Q_3(X^{12} - 3X^{11} - 3X^{10} + 3X^9 + 3X^5 - 3X^4 + 3X^3 + 3)$
    \item $K_8=\Q_3(X^{12} - 9X^{11} + 9X^9 - 9X^8 + 9X^7 - 12X^6 + 3X^3 + 9X^2 + 9X - 12)$
    \item $K_9=\Q_3(X^{12} + 9X^{11} + 9X^{10} - 3X^9 - 9X^8 - 9X^7 + 3X^6 + 9X^5 - 9X^4 + 6X^3 - 9X^2 - 9X + 12)$
    \item $K_{10}=\Q_3(X^{12} - 9X^{11} + 6X^9 + 9X^8 + 3X^6 + 9X^5 + 9X^4 + 3X^3 - 9X^2 + 9X + 3)$.
\end{itemize}
Looking at their respective Galois lattices we observe that $K_i^{\mathrm{un}}=K_j^{\mathrm{un}}$
if and only if $i\equiv j\bmod 5\Z$, so that there are in fact 5 fields of good reduction. Furthermore,
every one of these $5$ fields appear as the reduction field of some elliptic curve (see~\cite{Fr-Kr},Thm.17, (7)).
For $i=1,\ldots,10$ we let $L_i$ be the maximal totally ramified sub-extension of $K_i$, so that $K_i=L_i(\zeta_4)$.

Let $\alpha\in\mathcal{M}_{12}^{i,\epsilon}=\{\alpha\in L_i\vert \tau_4(\alpha)=-\alpha\text{ and } \tau_3(\alpha)=(\alpha+{(-1)}^{\epsilon+1}3)/(1+{(-1)}^{\epsilon}\alpha)\}$
for $i\in\{1,\ldots,5\}$ and $\epsilon\in\{ 0,1\}$. We let $K_0=\Q_3(\zeta_4)$ be the maximal unramified extension
of $\Q_3$ in $K_i$ which is independant of $i$. We denote by $\Dpc(12;0;i;\epsilon;\alpha)$ the
filtered $(\phi,\Gal({K_i}/{\Q_3}))$-module defined by:
\begin{itemize}
	\item $D=K_0e_1\oplus K_0e_2$
	\item $M_B(\tau_4)=\begin{pmatrix} \zeta_4 & 0 \\ 0 & \zeta_4^{-1}\end{pmatrix}$
	\item $M_B(\tau_3)=\begin{pmatrix} -\frac{1}{2} & \frac{{(-1)}^{\epsilon+1}3}{2} \\ \frac{{(-1)}^{\epsilon}}{2} & -\frac{1}{2}\end{pmatrix}$
	\item $\phi(e_1)=e_2;\ \phi(e_2)=-3e_1$
	\item $\omega(e_1)=e_1;\ \omega(e_2)=e_2$
	\item $\Fil^1D_{K_i}=(\alpha\otimes e_1 + 1\otimes e_2)K_i$.
\end{itemize}
For each $i\in\{1,\ldots,5\}$, $\epsilon\in\{0,1\}$ and $\alpha\in\mathcal{M}_{12}^{i,\epsilon}$, the filtered $(\phi,\Gal({K_i}/{\Q_3}))$-module $\Dpc(12;0;i;\epsilon;\alpha)$ satisfies conditions $(1)-(4)$ and is admissible.

\begin{proposition}
Let ${E}/{\Q_3}$ be an elliptic curve with $\dst(E)=12$ acquiring good reduction over $K_i$ for some $i\in\{1,\ldots,5\}$
and $\D=\Dcris[K_i](V_3(E))$.
There exists $\epsilon\in\{0,1\}$ and $\alpha\in\mathcal{M}_{12}^{i,\epsilon}$ such that $\D$ and $\Dpc(12;0;i;\epsilon;\alpha)$
are isomorphic as filtered $(\phi,\Gal({K_i}/{\Q_3}))$-modules. Moreover if $\epsilon,\epsilon^{\prime}\in\{0,1\}$
and $\alpha\in\mathcal{M}_{12}^{i,\epsilon}$, $\beta\in \mathcal{M}_{12}^{i,\epsilon^{\prime}}$,
then $\Dpc(12;0;i;\epsilon;\alpha)\simeq \Dpc(12;i;\epsilon^{\prime};\beta)$
if and only if $(\alpha,\epsilon) =(\beta,\epsilon^{\prime})$.
\end{proposition}
\begin{proof}
Let $D$ be the underlying $K_0$-vector space associated to $\D$. As usual, the inertia subgroup of ${K_i}/{\Q_3}$ injects
in $\Aut_{K_0}(D)$ and we identify $\tau_4$ and $\tau_3$ to their respective image. As in the quartic case we show that there is
a $K_0$-basis $B=(e_1,e_2)$ of $D$ such that:
\[
    \left\{
        \begin{array}{lll}
        M_B(\tau_4)=\begin{pmatrix} \zeta_4 & 0 \\ 0 & \zeta_4^{-1}\end{pmatrix}\\
        \omega(e_1)=e_1,\ \omega(e_2)=e_2 \\
        \phi(e_1)=e_2,\ \phi(e_2)=-3e_1 .
        \end{array}
    \right.
\]
The relations between $\tau_3$ and $\tau_4,\omega$ and $\phi$ implies that there is some $\epsilon^{\prime}\in\{ 0,1\}$ such that
\[
M_B(\tau_3)=\begin{pmatrix} -\frac{1}{2} & \frac{3{(-1)}^{\epsilon^{\prime}+1}}{2}\\ \frac{{(-1)}^{\epsilon^{\prime}}}{2} & -\frac{1}{2}\end{pmatrix}.
\]
A simple calculation shows that the $K_i$-lines of $D_{K_i}=K_i\otimes_{K_0}D$ stable by action of $\Gal({K_i}/{\Q_3})$ are of the
form
\[
(\alpha\otimes e_1 + 1\otimes e_2)K_i
\]
with $\alpha\in L_i$ satisfying the desired conditions.

Let $\epsilon,\epsilon^{\prime}\in\{0,1\}$ and $\alpha\in \mathcal{M}_{12}^{i,\epsilon}$, $\beta\in \mathcal{M}_{12}^{i,\epsilon^{\prime}}$.
Looking only at their underlying $(\phi,\Gal({K_i}/{\Q_3}))$-modules, we see that $\Dpc(12;0;i;\epsilon;\alpha)$
and $\Dpc(12;0;i;\epsilon^{\prime};\beta)$ are isomorphic if and only if $\epsilon=\epsilon^{\prime}$.
Now supposing $\epsilon=\epsilon^{\prime}$ and adding the filtration, we check that a morphism
between $\Dpc(12;0;i;\epsilon;\alpha)$ and $\Dpc(12;0;i;\epsilon;\beta)$ must be of the form $\lambda\Id$
with $\lambda\in\Q_3^{\times}$, so that necessarily $\alpha=\beta$.
\end{proof}

\begin{remark}
As in the cubic Abelian case, observe that $\Dpc(12;0;i;0;\alpha)$ and $\Dpc(12;0;i;1;\alpha)$ are unramified quadratic twists of
each other as $(\phi,\Gal({K_i}/{\Q_3}))$-modules.
\end{remark}


\section{Elliptic curves with given Tate module}\label{S5}

\subsection{Minimal Galois pairs}

Let $K/\Q_p$ be a finite Galois extension with residue field $\Fd_{p^s}$. A Galois pair for $K/\Q_p$ is a triple $(E_0,\Gamma,\nu)$ where $\Et_0/\Fd_p$ is an elliptic curve, $\Gamma$ a subgroup of $\Aut_{\Fd_{p^s}}(\Et)$, and
\[
\nu:\Gal(K/\Q_p)\longrightarrow\Aut_{\Fd_{p^s}}(\Et)\rtimes\Gal({\Fd_{p^s}}/{\Fd_p})
\]
is an antimorphism satisfying:
\begin{enumerate}
    \item $(\mathrm{pr}\circ\nu)(g)=g\bmod I(K/\Q_p)$ for all $g\in\Gal(K/\Q_p)$
    \item $\mathrm{Im}(\nu)=\Gamma\rtimes\Gal({\Fd_{p^s}}/{\Fd_p})$.
\end{enumerate}
Where $\Et=\Et_0\times_{\Fd_p}\Fd_{p^s}$. It is minimal if  $\nu$ is injective and $\Fd_{p^{s}}$ is minimal with respect to $\Gamma$. We refer to~\cite{Vo2}, \S3 for the properties of Galois pairs.

\begin{proposition}\label{minpairs}
    Every (unfiltered) $(\phi,\Gal({K}/{\Q_3}))$-module appearing in Table~\ref{tab:Table 1} comes from a minimal Galois pair for ${K}/{\Q_3}$.
\end{proposition}
    \renewcommand{\arraystretch}{1.3}
    \begin{table}
    \begin{center}
    \begin{tabular}{|c|c|c|l|l|}
    \hline
    Ram.\ ind. & $K$ & Trace & Minimal Galois pair & $(\phi,\Gal({K}/{\Q_3}))$-module\\
    \hline
    $e=3$ & $\Lng(\zeta_4)$ & $a_3=0$ & $\Et_0: y^2=x^3+x$ & $\Dpcng(3;0)$ \\
    & & & $\Gamma=\langle\tau\rangle,\text{ unique }$3$\text{-Sylow of}\Aut_{\Fd_9}(\Et)$ &\\
    & & & $\nu:\tau_3\mapsto\tau, \omega\mapsto f_{\sigma}$ &\\
    \cline{2-5}
    & $\Lg$ & $a_3=-3$ & $\Et_0: y^2=x^3-x+1$ & $\Dpcg(3;-3)$ \\
    & & & $\Gamma=\langle\tau\rangle,\text{ unique }$3$\text{-Sylow of}\Aut_{\Fd_3}(\Et_0)$ &\\
    & & & $\nu:\tau_3\mapsto\tau$ &\\
    \cline{3-5}
    & & $a_3=0$ & $\Et_0: y^2=x^3-x$ & $\Dpcg(3;0)$\\
    & & & $\Gamma=\langle\tau\rangle,\text{ unique }$3$\text{-Sylow of}\Aut_{\Fd_3}(\Et_0)$ &\\
    & & & $\nu:\tau_3\mapsto\tau$ &\\
    \cline{3-5}
    & & $a_3=3$ & $\Et_0: y^2=x^3-x-1$ & $\Dpcg(3;3)$ \\
    & & & $\Gamma=\langle\tau\rangle,\text{ unique }$3$\text{-Sylow of}\Aut_{\Fd_3}(\Et_0)$ &\\
    & & & $\nu:\tau_3\mapsto\tau$ &\\
    \hline
    $e=12$ & $K_i$ & $a_3=0$ & $\Et_0: y^2=x^3+x$ & $\Dpc(12;0;i;1)$ \\
    & & & $\Gamma=\Aut_{\Fd_9}(\Et)$ & \\
    & & & $\nu:\omega\mapsto f_{\sigma}$ & \\
    \hline
    \end{tabular}
    \end{center}
    \caption{\label{tab:Table 2}Minimal Galois pairs for $e=3$ and $e=12$.}
    \end{table}
\begin{proof}
We only treat the wild cases that are not quadratic twists, i.e.\ the cubic and dodecic ones. Let us denote $\Et=\Et_0\times_{\Fd_3}\Fd_9$. The minimal Galois pairs are given in Table~\ref{tab:Table 2} below. It is not hard to see that $\nu$ is injective and the field of definition of $\Gamma$
is minimal. Each of these object gives rise to a $(\phi,\Gal({K}/{\Q_3}))$-module which is necessarily in our list by
construction (they have the right Frobenius and Galois action). Except for the non abelian cubic case,
there are always two isomorphisms classes of $(\phi,\Gal({K}/{\Q_3}))$-modules in our list (see section~\ref{S4}). We have only checked that one of them comes from
a Galois pair but in fact both do since they are unramified quadratic twists of each other.

\end{proof}

\begin{remark}
    When $a_3(\Et_0)=\pm3$ a Galois pair for $K={\Lng(\zeta_4)}/{\Q_3}$ is never minimal because $\Aut_{\Fd_9}(\Et)$ is too small
    compared to $\Gal({K}/{\Q_3})$. It is another way to see why those traces are absent from our list in that case.
\end{remark}

\subsection{A complete classification}

To every $3$-adic potentially crystalline representation $V$ of $\Gal({\Qal_3}/{\Q_3})$ corresponds a weakly admissible
filtered $(\phi,\Gal({K}/{\Q_3}))$-module $\Dcris[K](V)$. This association is functorial in a fully faithful way. In this section,
we will show that every object described in Table~\ref{tab:Table 1} comes from an elliptic curve over $\Q_3$ with potential
good reduction. It turns out that we can use the same tools and ingredients as M. Volkov in her treatment of the
tame case (see~\cite{Vo2}).

\begin{theorem*}
    Let $\D$ be one of the filtered $(\phi,\Gal({K}/{\Q_3}))$-module in Table~\ref{tab:Table 1}. There exists an elliptic curve ${E}/{\Q_3}$ such that $\D\simeq\Dcris[K](V_3(E))$.
\end{theorem*}
\begin{proof}
    We sketch the proof, following the arguments of~\cite{Vo2} in Thm.5.7. The $\phi_0$-module $D_0$ comes
    from an elliptic curve ${\Et_0}/{\Fd_3}$ with the right Frobenius (via the Dieudonné module of its $p$-divisible group).
    Let $\Et=\Et_0\times_{\Fd_3}k$. Since $\D=\Dcris[K](V)$ for some crystalline representation $V$ of $\Gal(\Qal_3/K)$
    with Hodge-Tate weights $(0,1)$, there exists a $p$-divisible group $\mathcal{G}/\Oc_K$ lifting $\Et(p)/k$ with Tate module
    isomorphic to $V$ (see~\cite{Br},Thm.5.3.2).
    By the Serre-Tate Theorem, the triple $(\mathcal{G},\Et(p),\tilde{\mathcal{G}}\tilde{\longrightarrow}\Et(p))$ determines an elliptic curve $E/K$ with good reduction (i.e.\ an elliptic scheme over $\mathcal{O}_K$) such that $V_3(E)\simeq V$.
    Finally, a minimal Galois pair $(\Et_0,\Gamma,\nu)$ for ${K}/{\Q_3}$ (which always exists in the tame case
    by~\cite{Vo2},Thm.4.11 and in the wild case by Prop.~\ref{minpairs}) furnishes the necessary descent datum
    to obtain an elliptic curve ${E_0}/{\Q_3}$ such that $E=E_0\times_{\Q_3}K$ and $V\simeq V_3(E_0)$.
\end{proof}

\end{document}